\documentclass[12pt,reqno,a4paper]{amsart}
\usepackage{hyperref}
\usepackage{enumerate}
\usepackage[left=0.7in, right=0.7in, top= 0.83in, bottom=1.63in]{geometry}
\usepackage{amsmath,amssymb,amsthm,amsfonts}
\usepackage[dvipsnames]{xcolor}
\usepackage{tikz}

\usepackage{nicematrix}
\usepackage{caption}
\usetikzlibrary{matrix, decorations.pathreplacing, calc, fit}
\usepackage{float} 

\usetikzlibrary{positioning}

	\usepackage{verbatim}
\usetikzlibrary{arrows,shapes}
\usetikzlibrary{positioning, 
	quotes}

\usepackage{graphicx}
\usetikzlibrary{graphs, graphs.standard}

\tikzset{
	modal/.style={>=stealth’,shorten >=1pt,shorten <=1pt,auto,node distance=1.5cm,
		semithick},
	world/.style={circle, draw,minimum size=.1cm,fill=gray!15},
	point/.style={circle,draw,inner sep=0.3mm,fill=black},
	circ/.style={circle,draw,inner sep=0.1mm,fill=white},
	reflexive above/.style={->,loop,looseness=7,in=120,out=60},
	reflexive below/.style={->,loop,looseness=7,in=240,out=300},
	reflexive left/.style={->,loop,looseness=7,in=150,out=210},
	reflexive right/.style={->,loop,looseness=7,in=30,out=330}
}
\pgfdeclarelayer{background}
		\pgfsetlayers{background,main}
		
		\tikzstyle{vertex}=[circle,fill=black!25,minimum size=14pt,inner sep=0pt]
		\tikzstyle{selected vertex} = [vertex, fill=red!24]
		\tikzstyle{edge} = [draw,thick,-]
		\tikzstyle{weight} = [font=\small]
		\tikzstyle{selected edge} = [draw,line width=5pt,-,red!50]
		\tikzstyle{ignored edge} = [draw,line width=5pt,-,black!20]
\usetikzlibrary{shapes}
\usetikzlibrary{plotmarks}
\usetikzlibrary{arrows}
\usetikzlibrary{positioning}
\theoremstyle{definition}
\newtheorem{defn}{Definition}[section]

\theoremstyle{plain}
\newtheorem{thm}[defn]{Theorem}
\newtheorem{lem}[defn]{Lemma}
\newtheorem{prop}[defn]{Proposition}

\newtheorem{conjecture}[defn]{Conjecture}
\newtheorem{fact}[defn]{Fact}
\newtheorem{claim}[defn]{Claim}

\newtheorem{question}[defn]{Question}

\title[AVD-Total Coloring]
{AVD-total colorings of Subdivision Graphs, Joins, and Deleted Lexicographic Products}

\author{Amitayu Banerjee}
\address{E\"otv\"os Lor\'and University, Budapest, Hungary}
\email{banerjee.amitayu@gmail.com}

\date{}

\makeatletter
\@namedef{subjclassname@2020}{\textup{2020} Mathematics Subject Classification}
\makeatother
\begin{document}

\begin{abstract}
An {\em adjacent vertex distinguishing (AVD) total coloring} of a graph $G=(V(G),E(G))$ is a proper total coloring $f$ such that for every edge $uv\in E(G)$, we have
$
C_G(u)\neq C_G(v),
$
where $C_G(u)=\{f(u)\}\cup \{f(uv): uv\in E(G)\}$. 
The 
{\em AVD-total chromatic number}
 $\chi''_{a}(G)$ of $G$ is the least integer $k$ such that there exists an AVD-total coloring of $G$ using $k$ colors.
The {\em AVD-total coloring conjecture (AVD-TCC)} asserts that $\chi''_{a}(G)\leq \Delta(G)+3$ for every simple graph $G$, where $\Delta(G)$ is the maximum degree of $G$.
The {\em central graph} of $G$ is obtained from the subdivision graph of $G$ by joining all pairs of non-adjacent vertices of $G$. 

Panda, Verma, and Keerti (AKCE Int. J. Graphs
Comb. 17: 820–825, 2020) asked whether the central graph of any graph satisfies the AVD-TCC. 
We use idempotent commutative Latin squares and a theorem of Borodin, Kostochka, and Woodall to prove the 
AVD-TCC for central graphs of all connected graphs with order at least five and minimum degree at least two.
We also determine the AVD-total chromatic number of subdivision graphs of connected graphs and verify the AVD-TCC for certain classes of joins of graphs and deleted lexicographic products.
\end{abstract}

\maketitle
\section{Introduction and Definitions}

All graphs considered in this paper are finite, undirected, and simple.
The sets $V(G)$ and $E(G)$ are the vertex and edge sets of a graph $G$.
For $v\in V(G)$, let $deg_G(v)$ denote the degree of $v$, and let $\Delta(G)$ and $\delta(G)$ be the maximum and minimum degrees of $G$.
For $X\subseteq V(G)$, let $G[X]$ denote the subgraph of $G$ induced by $X$.
A \emph{proper total coloring} of  $G$ is a coloring of $V(G)\cup E(G)$  such that no two adjacent or incident elements receive the same color.
The minimum number of colors required for a proper total coloring of $G$ is called the \emph{total chromatic number} of $G$, denoted by $\chi''(G)$.
Behzad \cite{Beh1965} and Vizing \cite{Viz1968} independently posed the following conjecture.

\begin{conjecture}[Total Coloring Conjecture (TCC)]\label{Conjecture 1.1}
For any graph $G$, $\chi''(G) \leq \Delta(G) + 2$.
\end{conjecture}

The {\em adjacent vertex distinguishing (AVD) total chromatic number}
of $G$, denoted by $\chi''_{a}(G)$, is the least integer $k$ such that $G$ admits a proper total coloring $f$ with $k$ colors that satisfy $C_G(u)\neq C_G(v)$ for every edge $uv\in E(G)$, where
$C_G(u)=\{f(u)\}\cup\{f(uw):uw\in E(G)\}$.
Zhang et al.~\cite{ZCYLW2005} introduced adjacent vertex distinguishing total colorings and posed the following conjecture.

\begin{conjecture}\label{Conjecture 1.2}(AVD-TCC). For any graph $G$, $\chi''_{a}(G) \leq \Delta(G) +3$.
\end{conjecture}

We say that a graph $G$ is of {\em AVD-Type $i$} if $\chi''_{a}(G) = \Delta(G) + i$ for $i\in\{1,2,3\}$ and of {\em Type $j$} if $\chi''(G) = \Delta(G) + j$ for $j\in\{1,2\}$. 
The Conjecture~\ref{Conjecture 1.2} has been verified for graphs with $\Delta(G)\in\{3,4\}$~\cite{Wang2007,Che2008,LLLM2017}, 4-regular graphs~\cite{PR2014}, hypercubes~\cite{CG2009}, complete equipartite graphs~\cite{LCM2015}, indifference graphs~\cite{PM2010}, split graphs~\cite{VFP2022}, and several other graph classes.

\begin{defn}\label{Definition 1.3}
The {\em chromatic index} of a graph $G$, denoted by $\chi'(G)$, is the least integer $k$ such that $G$ admits a proper edge coloring with $k$ colors.
If $\chi'(G)=\Delta(G)$, then $G$ is called a class-I graph.  
The {\em vertex-distinguishing chromatic index} of $G$, denoted by {$\chi'_{vd}(G)$}, is the least integer $k$ such that $G$ admits a proper edge coloring $f$ with $k$ colors that satisfy $D_G(u)\neq D_G(v)$ for every 
pair of distinct vertices $u, v \in V(G)$,
where
$D_G(u)=\{f(uw):uw\in E(G)\}$.
Let $P_n$, $C_n$, and $K_n$ denote the path, cycle, and complete graph on $n$ vertices, respectively, and $K_{n,m}$ be a complete bipartite graph with bipartitions of sizes $n$ and $m$. 
\end{defn}

Burris and Schelp \cite{BS1997} and, independently, Aigner, Triesch, and Tuza \cite{ATT1990} introduced the concept of vertex-distinguishing edge coloring. Since then, several bounds on the vertex-distinguishing chromatic index have been established for various classes of graphs.

\begin{defn}\label{Definition 1.5}
The \emph{subdivision graph} $S(G)$ of $G$ is obtained from $G$ by replacing each edge $xy\in E(G)$ with the two edges $xw_{xy}$ and $w_{xy}y$, where $w_{xy}$ is a new vertex (called the subdivision vertex corresponding to the edge $xy$).
The {\em central graph $C(G)$} of $G$ is obtained from $S(G)$ by adding all edges of $\overline{G}$, where $\overline{G}$ is the complement of $G$.
\end{defn}

The central graph is introduced by Vernold~\cite{Ver2007}. These graphs have been well studied; see, for example, \cite{AG2016, AM2022, BFKP2026, CSW2020, Eff2017, KK2021, KM2019, PVK2020, SM2017}.

\begin{defn}\label{Definition 1.6}
The {\em join} of two graphs $G$ and $H$, denoted by $G\vee H$, is obtained by taking $G$ and $H$ and adding
edges between every vertex of $G$ and every vertex of $H$.
The {\em deleted lexicographic product} of $G$ and $H$, 
denoted by $D_{lex}(G, H)$, is a graph with the vertex set 
$V(G) \times V(H)$ and the edge set 
$
E = \left\{ ((g_1, h_1), (g_2, h_2)) : (g_1, g_2) \in E(G) \text{ and } h_1 \neq h_2, 
\text{ or } (h_1, h_2) \in E(H) \text{ and } g_1 = g_2 \right\}.
$
\end{defn}

\subsection{Motivation} 
Kavaskar and Sukumaran \cite{KS2024} studied the total colorings of the joins of graphs and proved that if $G$ satisfies the TCC then the join $G\vee G$ satisfies the TCC and the join of two Type 1 graphs having the same order satisfies the TCC.
The deleted lexicographic product was introduced by Miklavi\v{c} and Milani\v{c} \cite{MM2011} and further studied by Frelih and Miklavi\v{c} \cite{FM2013}.
The total chromatic number for certain classes of these products was studied in \cite{SGS2024, SGS2022, VGS2018}.
In particular, Vignesh, Geetha, and Somasundaram \cite{VGS2018} proved that for any class-I graph $G$ and any graph $H$ with at least $3$ vertices, $D_{lex}(G, H)$ satisfies the TCC.
Motivated by these results, we study AVD-total colorings of joins of graphs and
the deleted lexicographic product of graphs.
Sudha and Manikandan~\cite{SM2017} studied the total chromatic number of the central graphs of stars, paths, and cycles.
Panda, Verma, and Keerti~\cite{PVK2020} investigated the AVD-total chromatic number of the central graphs of complete graphs, stars, paths, and cycles
and posed the following open question.

\begin{question}\label{Question 1.7}
(Panda et al. \cite{PVK2020}) Does AVD-TCC hold for the central graph of any graph?
\end{question}

The authors of \cite{PVK2020} also proposed investigating the classification of AVD-total colorings for central graphs. In this paper, we establish the AVD-TCC for central graphs of graphs with order $n\geq 5$  and minimum degree at least two, and determine when these central graphs are AVD-Type 2.

\subsection{Main Results}

In Section 2, we study the classification problem for the AVD-total chromatic number of subdivision graphs of connected graphs. 

\begin{enumerate}
    \item (Proposition \ref{Proposition 2.5}, Theorem \ref{Theorem 2.6}) For every connected graph $G$, the subdivision graph $S(G)$ is of AVD-Type 1 if and only if 
    $\Delta(G)\geq 3$ or $G\cong P_2$, and is of AVD-Type 2 otherwise.
\end{enumerate}

In Section 3, we prove the AVD-TCC for certain classes of deleted lexicographic products.

\begin{enumerate}
    \item (Theorem~\ref{Theorem 3.4}) 
    If $G$ is a class-I graph, $H$ is a graph of order $n\geq 3$ such that $\chi'_{vd}(H)\leq \Delta(H)+i$ for $i\in \{0,1,2\}$, then $\chi''_{a}(D_{lex}(G, H))\leq \Delta(D_{lex}(G, H)) +i+1$.
    \item (Theorem~\ref{Theorem 3.6}) If $G$ is a graph of order $m$ such that $m\geq \Delta(G)+3$ and $\chi''_{a}(G)\leq \Delta(G)+2$, then $\chi''_{a}(D_{lex}(P_n, G))\leq \Delta(D_{lex}(P_n, G))+2$ for every $n\geq 3$. 
    
\end{enumerate} 

In Section 4, we apply Latin squares to study the AVD-total chromatic number of the join of graphs and the central graphs of regular graphs.

\begin{enumerate}
     \item (Theorem~\ref{Theorem 4.5}) Let $G_1$ and $G_2$ be graphs of orders $n$ and $m$, respectively, with $m>n+1$ and $\Delta(G_2)\leq \Delta(G_1)$. If $G_1$ satisfies the TCC then $G_1\vee G_2$ satisfies the AVD-TCC.

    \item (Theorem~\ref{Theorem 4.8}) If $G_1$ and $G_2$ are graphs of order $n$ such that $\Delta(G_2)\leq \Delta(G_1)$, and $G_1$ is a Type 1 graph, then $G_1\vee G_2$ satisfies the AVD-TCC. 

    \item (Theorem~\ref{Theorem 4.9}) If $G\not\cong K_n$ is a connected graph of order $n\geq 5$ with $\delta(G)\geq 2$, then $C(G)$ is AVD-Type 2 when $n$ is even, and $C(G)$ satisfies the AVD-TCC when $n$ is odd.

\end{enumerate}

Panda et al.~\cite{PVK2020} proved that the TCC holds for the central graph of any graph. 
In Section~5, we provide an alternative proof of this result. In Proposition \ref{Proposition 5.4}, 
we prove the following results for graphs $G_1$ and $G_2$ of order at least 3:

\begin{enumerate}
    \item If $C(G_{1})$ and $C(G_{2})$ satisfy the AVD-TCC then $C(G_{1} \vee G_{2})$ satisfies the AVD-TCC.
    \item If $G_{1}$ and $G_{2}$ have the same order, then $C(G_{1} \vee G_{2})$ satisfies the AVD-TCC.
    \item If $G_{1}\in\{P_{n},C_{n},K_{n}\}$ and
    $G_{2}\in\{P_{m},C_{m},K_{m}\}$, then
    $C(G_{1}\vee G_{2})$ satisfies the AVD-TCC for any integers $m,n\geq3$.
\end{enumerate}

\section{Subdivision graphs}
The following result shows that every bipartite graph $G$ admits an AVD-total coloring using at most $\Delta(G)+2$ colors.

\begin{thm}[{\cite[Proposition 1.1]{Che2008}}]\label{Theorem 2.1}
If $G$ is a bipartite graph, then $\chi''_{a}(G)\leq \Delta(G) + 2$.
\end{thm}

Since the subdivision graph $S(G)$ is bipartite, Theorem~\ref{Theorem 2.1} yields
$\chi''_{a}(S(G)) \leq \Delta(S(G)) + 2$.
In this section, we determine when $S(G)$ is AVD-Type~1.
We recall the following known results.

\begin{fact}[K\"{o}nig's Theorem]\label{Fact 2.2}
If $G$ is a bipartite graph, then $\chi'(G) = \Delta(G)$.
\end{fact}

\begin{fact}[{\cite[Theorem 2.1]{ZCYLW2005}}]\label{Fact 2.3}
If $n \geq 4$, then $\chi''_{a}(C_n) = 4$.
\end{fact}

\begin{fact}[{\cite[Lemma 2.1]{ZCYLW2005}}]\label{Fact 2.4}
If $n = 2$ or $3$, then $\chi''_{a}(P_n) = 3$, and if $n \geq 4$, then $\chi''_{a}(P_n) = 4$.
\end{fact}

\begin{prop}\label{Proposition 2.5}
Let $G$ be a connected nontrivial graph with $\Delta(G) \leq 2$. Then 
\[
\chi''_{a}(S(G)) = 
\begin{cases} 
3 & \text{if } G \cong P_2, \\ 
4 & \text{otherwise}. 
\end{cases}
\]
\end{prop}

\begin{proof}
By assumption, it follows that
$G\cong P_n$ for some $n\geq 2$ or $G\cong C_n$ for some $n\geq 3$. 

\textbf{Case (i):} If $G \cong P_n$ for $n\geq 2$, then $S(G) \cong P_{2n-1}$. 
\begin{itemize}
    \item If $n=2$, then $S(G) \cong P_3$. By Fact~\ref{Fact 2.4}, $\chi''_{a}(S(G))=\chi''_{a}(P_3) = 3$.
    \item If $n \geq 3$, then $2n-1 \geq 5$. Thus, by Fact~\ref{Fact 2.4}, $\chi''_{a}(S(G))=\chi''_{a}(P_{2n-1}) = 4$.
\end{itemize}

\textbf{Case (ii):} If $G \cong C_n$ for $n \geq 3$, then $S(G) \cong C_{2n}$ where $2n \geq 6$. By Fact~\ref{Fact 2.3}, we have $\chi''_{a}(S(G))=\chi''_{a}(C_{2n}) = 4$.
\end{proof}

In $S(G)$, let $V_1=V(G)$ and let $V_2$ be the set of vertices obtained by subdividing each edge of $G$.

\begin{thm}\label{Theorem 2.6}
If $G$ is a connected graph and $\Delta(G)\geq 3$ then $S(G)$ is $AVD$-Type 1.
\end{thm}

\begin{proof}
We have $\Delta(G)\geq3$, $deg_{S(G)}(v)=deg_G(v)$ for every $v\in V_1$, and $deg_{S(G)}(v)=2$ for every $v\in V_2$. Hence, $\Delta(S(G))=\Delta(G)$.
We have $\Delta(S(G))+1\leq \chi''(S(G))\leq \chi''_{a}(S(G))$. Thus, it suffices to prove that $\chi''_{a}(S(G))\leq\Delta(G)+1$.
We construct an AVD-total coloring $f$ of $S(G)$ with colors from the set $S=\{1,...,\Delta(G)+1\}$. 
By Fact \ref{Fact 2.2}, there is a proper edge coloring $g$ of $S(G)$ with colors from the set $\{1,...,\Delta(G)\}$
as $S(G)$ is a bipartite graph. Let $f(e)=g(e)$ for any $e\in E(S(G))$.
Let $f(v)=\Delta(G)+1$ for all $v\in V_{1}$.
Fix a subdivision vertex $v'=w_{vw}\in V_{2}$.

\begin{figure}[h]
\centering
\begin{tikzpicture}[
scale=0.54,
every node/.style={circle, draw, inner sep=0.8pt, minimum size=3mm},
dotnode/.style={draw=none}
]

\draw[rounded corners] (-4,-1.2) rectangle (-1,1.4);
\draw[rounded corners] (1,-1.2) rectangle (4,1.4);

\node[draw=none] at (-4.5,1) {$V_1$};
\node[draw=none] at (4.5,1) {$V_2$};

\node (v) at (-3,0.7) {$v$};
\node (dots1) [dotnode] at (-3,0) {$\cdots$};
\node (w) at (-3,-0.7) {$w$};

\node (vw) at (3,0) {$w_{vw}$};
\node (dots2) [dotnode] at (3,-1) {$\cdots$};

\draw (v) -- (vw);
\draw (w) -- (vw);

\end{tikzpicture}
\vspace{-2mm}
\caption{\em Subdivision graph $S(G)$ with bipartitions $V_1$ and $V_2$}
\end{figure}

Let 
$R_{v'}=\{1,...,\Delta(G)+1\}\backslash\{f(v),  f(vv'), f(wv')\}=\{1,...,\Delta(G)\}\backslash\{  f(vv'), f(wv')\}$. 
Since $\Delta(G)\geq 3$, we have $R_{v'}\neq \emptyset$. Let $f(v')$ be any color from $R_{v'}$.  
From the construction, it follows that $f$ is a proper total coloring of $S(G)$.
We show that $f$ is an AVD-total coloring of $S(G)$.
Since $S(G)$ is bipartite with bipartitions $V_1$ and $V_2$, all edges in $S(G)$ have one endpoint in $V_1$ and the other in $V_2$. 
Let $vv'\in E(S(G))$, where $v\in V_1$ and $v'=w_{vw}\in V_2$ is the subdivision vertex corresponding to some edge $vw\in E(G)$.
Since $f(v)=\Delta(G)+1\in C_{S(G)}(v)$ and $\Delta(G)+1\notin C_{S(G)}(v')$ (as $f(v')$, $f(vv')$, and $f(wv')$ are different from $\Delta(G)+1$), we have $C_{S(G)}(v)\neq C_{S(G)}(v')$.
Thus, for every edge $uv\in E(S(G))$, we have $C_{S(G)}(u)\neq C_{S(G)}(v)$, and so $f$ is an AVD-total coloring.
\end{proof}
\section{Deleted Lexicographic products}

\begin{fact}[\cite{Yap1996}]\label{Fact 3.1}
Let $G$ be a graph with $n$ vertices. If $\Delta(G) \geq \frac34 n$, then $G$ satisfies the TCC.
\end{fact}

\begin{prop}\label{Proposition 3.2}
Let $H$ be a connected graph of order $m\ge 5$. If
$(m-4)(n-4)\ge 4$,
then $D_{lex}(K_n,H)$ satisfies the TCC.
\end{prop}

\begin{proof}
The graph $D_{lex}(K_n,H)$ has order $mn$. 
By assumption, $\Delta(H)\ge2$ and $(m-4)(n-4)\ge4$. Hence
$
\Delta(D_{lex}(K_n,H))
=
\Delta(H)+(n-1)(m-1)\ge
2+(n-1)(m-1)\ge\frac34 mn.
$
By Fact \ref{Fact 3.1},
$
D_{lex}(K_n,H)
$
satisfies the TCC.
\end{proof}

\begin{fact}[\cite{Yap1996}]\label{Fact 3.3}
For any integer $n \geq 3$ there exists an $n$ edge coloring of $K_{n,n}$ such that
$K_{n,n}$ has a perfect matching with $n$ distinct colors.
\end{fact} 

\begin{thm}\label{Theorem 3.4}
    Let $G$ be a class-I graph, $H$ be a graph of order $n\geq 3$ such that $\chi'_{vd}(H)\leq \Delta(H)+k$ for $k\in \{0,1,2\}$. Then $\chi''_{a}(D_{lex}(G, H))\leq \Delta(D_{lex}(G, H)) +k+1$.
\end{thm}

\begin{proof}
    Fix $k\in \{0,1,2\}$. The maximum degree of $D_{lex}(G, H)$ is  $\Delta(D_{lex}(G, H))=\Delta(H)+\Delta(G)(n-1)$. 
    We divide $\Delta(H) + \Delta(G)(n-1) + k+1$ colors into the following disjoint sets:
\begin{align*}
C_0 &= \{a_1^0, a_2^0, \dots, a_{\Delta(H)+k}^0\}, \\
C_i &= \{a_1^i, a_2^i, \dots, a_{n-1}^i\}, 2\leq i\leq \Delta(G),\\
C_{\Delta(G)+1} &= \{1,...,n\}.
\end{align*}

\begin{figure}[htbp]
\vspace{-2mm}
    \centering
    \begin{tikzpicture}[
        scale=0.69, transform shape,
        vertex/.style={circle, draw, fill=black, inner sep=0pt, minimum size=4.5pt},
        gvertex/.style={circle, draw, fill=white, thick, inner sep=0pt, minimum size=5.5pt},
        setcontainer/.style={draw, rectangle, rounded corners=6pt, dashed, black, inner sep=12pt},
        gcontainer/.style={draw, rectangle, rounded corners=10pt, thick, black, inner sep=10pt},
        hcontainer/.style={draw, rectangle, rounded corners=6pt, black, inner sep=8pt},
        joinstyle/.style={draw, thin, black}
    ]

        \foreach \i in {1,...,5} {
            \node[vertex] (u\i) at (0, \i*0.75 + 1.0) {};
        }
        \node (dotsY) at (-0.2, 3.75) {$\vdots$};
        
        \foreach \j in {1,...,5} {
            \node[vertex] (w\j) at (2.6, \j*0.75 + 1.0) {};
        }
        \node (dotsY) at (2.8, 3.75) {$\vdots$};

        \node[hcontainer, fit=(u1) (u5), label=above:$H_{v_k}$] (Hvk) {};
        \node[hcontainer, fit=(w1) (w5), label=above:$H_{v_{l}}$] (Hvk1) {};

        \foreach \i in {1,...,5} {
            \foreach \j in {1,...,5} {
                \ifnum\i=\j\else
                    \draw[joinstyle] (u\i) -- (w\j);
                \fi
            }
        }

        \node[gvertex, label=below:$v_1$] (vr) at (-2.0, 0.4) {};
        \node (dotsX) at (-1.0, 0.4) {$\cdots$};
        \node[gvertex, label=below:$v_k$] (vk) at (0, 0.4) {};
        
        \node[gvertex, label=below:$v_{l}$] (vk1) at (2.6, 0.4) {};
        \node (dotsY) at (3.6, 0.4) {$\cdots$};
        \node[gvertex, label=below:$v_{m}$] (vr1) at (4.6, 0.4) {};

        \draw[black] (vk) -- (vk1);

        \draw[dotted, thick, ->, black] (vk) -- (Hvk);
        \draw[dotted, thick, ->, black] (vk1) -- (Hvk1);

        \node[gcontainer, fit=(vr) (vk) (vk1) (vr1), label=left:$G$] (G) {};

\end{tikzpicture}
    \caption{\em Edges between $H_{v_k}$ and $H_{v_{l}}$, two copies of $H$ in $D_{lex}(G,H)$.}
    \label{Figure 2}
\end{figure}

Let $H_{v_i}$ be the copy of $H$ with respect to the vertex $v_i\in V(G)$.
Define a total coloring $c:V(D_{lex}(G, H))\cup E(D_{lex}(G, H))\rightarrow 
C_0\cup \bigcup_{t=2}^{\Delta(G)+1} C_t$ of $D_{lex}(G, H)$ as follows:
\begin{enumerate}
    \item For each $v_r\in V(G)$, we consider a vertex-distinguishing edge coloring of $H_{v_r}$ with colors from $C_0$ so that the corresponding elements of $H_{v_r}$ receive the same colors. In particular, let
$
\varphi:E(H)\rightarrow C_0
$
be a vertex-distinguishing edge coloring of \(H\). For every \(v_r\in V(G)\), define
$c((v_r,x)(v_r,y))=\varphi(xy)$ for all $xy\in E(H)$.    

\item We assign distinct colors from $C_{\Delta(G)+1}$ to the vertices of each
copy of $H$.

    \item Since $G$ is a class-I graph, there exists a proper edge coloring $l:E(G)\rightarrow \{1,...,\Delta(G)\}$ of $G$. For every edge $v_kv_l\in E(G)$ with $l(v_kv_l)=i$, let
$B(H_{v_k},H_{v_l})$ be the bipartite graph with bipartitions
$H_{v_k}$ and $H_{v_l}$.

    \begin{itemize}    
        \item Fix $i\in \{2,...,\Delta(G)\}$ and $l(v_kv_l)=i$.
     Since
$B(H_{v_k},H_{v_l})\cong K_{n,n}-M$, where $M$ is a perfect matching,
every vertex of $B(H_{v_k},H_{v_l})$ has degree $n-1$. Hence, by
Fact \ref{Fact 2.2},
\[
\chi'(B(H_{v_k},H_{v_l}))
=
\Delta(B(H_{v_k},H_{v_l}))
=
n-1
=
|C_i|.
\]
Therefore, there exists a proper edge coloring
$
\phi_{kl}:E(B(H_{v_k},H_{v_l}))\rightarrow C_i
$.
We define $c(e)=\phi_{kl}(e)$ for every $e\in E(B(H_{v_k},H_{v_l}))$.

        \item If $i=1$, then for every edge $v_kv_l\in E(G)$ with $l(v_kv_l)=i$, we color edges in $B(H_{v_k},H_{v_l})$.
        By Fact \ref{Fact 3.3},  there exists a proper edge  coloring $\psi_{kl}$ that colors the join edges between $H_{v_k}$ and $H_{v_l}$ in the join graph $H_{v_k}\vee H_{v_l}$ with colors from $C_{\Delta(G)+1}$ such that the edges joining corresponding vertices are assigned $n$ different colors.
        We observe that in $B(H_{v_k},H_{v_l})$, the edges joining the corresponding vertices of $H_{v_k}$ and $H_{v_l}$ are absent, and the endpoints of these join edges are colored by the color of these absent edges in (2). We define $c(e)=\psi_{kl}(e)$ for every $e\in E(B(H_{v_k},H_{v_l}))$.
\end{itemize}
\end{enumerate}

\begin{claim}\label{Claim 3.5}
The coloring $c$ is an AVD-total coloring of $D_{lex}(G,H)$.
\end{claim}

\begin{proof}
By construction, the coloring $c$ is a proper total coloring of $D_{lex}(G,H)$. Let $xy\in E(D_{lex}(G, H))$.
We show that $
C_{D_{lex}(G, H)}(x)\neq C_{D_{lex}(G, H)}(y).
$

\textbf{Case (i).}
Let $x=(v_i,a)\in V(H_{v_i})$ and $y=(v_i,b)\in V(H_{v_i})$ for some $a,b\in V(H)$ and $v_{i} \in V(G)$.
Then $D_{H_{v_i}}(x)=\{c(ux):ux\in E(H_{v_i})\}\neq \{c(uy):uy\in E(H_{v_i})\}= D_{H_{v_i}}(y)$ as we considered a vertex-distinguishing edge coloring of $H_{v_i}$ in (1).
Since
$
D_{H_{v_i}}(x),\; D_{H_{v_i}}(y)\subseteq C_0,
$
and
\[
C_{D_{lex}(G, H)}(x)\setminus D_{H_{v_i}}(x),\;
C_{D_{lex}(G, H)}(y)\setminus D_{H_{v_i}}(y)
\subseteq \bigcup_{t=2}^{\Delta(G)+1} C_t,
\]
where
$
\left(\bigcup_{t=2}^{\Delta(G)+1} C_t\right)\cap C_0=\emptyset,
$
it follows that
$
C_{D_{lex}(G, H)}(x)\neq C_{D_{lex}(G, H)}(y).
$

\textbf{Case (ii).} Let $x=(v_i,a)\in V(H_{v_i})$ and $y=(v_j,b)\in V(H_{v_j})$ for some $a,b\in V(H)$ and $v_i,v_j\in V(G)$ such that $i\neq j$. Let
$
x'=(v_j,a)
$
be the vertex of $H_{v_j}$ corresponding to $x$. Since the edges of every copy of $H$ is colored identically,
$
D_{H_{v_i}}(x)=D_{H_{v_j}}(x').
$
As $H_{v_j}$ is vertex-distinguishing edge colored,
$
D_{H_{v_j}}(x')\neq D_{H_{v_j}}(y).
$
Hence,
$
D_{H_{v_i}}(x)\neq D_{H_{v_j}}(y).
$
Applying the same argument as in Case (i), we conclude that $C_{D_{lex}(G, H)}(x)\neq C_{D_{lex}(G, H)}(y)$.
\end{proof}
\end{proof}

\begin{thm}\label{Theorem 3.6}
If $G$ is a graph of order $m$ such that $m\geq \Delta(G)+3$ and $\chi''_{a}(G)\leq \Delta(G)+2$, then $\chi''_{a}(D_{lex}(P_n, G))\leq \Delta(D_{lex}(P_n, G))+2$ for every  $n\geq 3$. 
\end{thm}

\begin{proof}
Fix an integer $n\geq 3$. Let $V(P_n)=\{v_1,...,v_n\}$ be an enumeration of the set of vertices of $P_n$. In $D_{lex}(P_n, G)$, there are $n$ copies of $G$ and the maximum degree of $D_{lex}(P_n, G)$ is $\Delta(G)+2(m-1)$. 
Let $G_{v_i}$ denote the copy of $G$ with respect to the vertex $v_i\in V(P_n)$.
Since $m-1\geq\Delta(G)+2$, we split $\Delta(G)+2(m-1)+2$ colors into the following disjoint sets:
$S_1=\{a_1,...,a_{\Delta(G)+2}\}$, $S_2=\{b_1,...,b_{\Delta(G)+2}\}$, $S'_{2}=\{b_{\Delta(G)+3},...,b_{m-1}\}$, $S_3=\{c_1,...,c_{\Delta(G)+2}\}$, $S'_{3}=\{c_{\Delta(G)+3},...,c_{m-1}\}$.\footnote{If $m-1=\Delta(G)+2$ then $S'_2=\emptyset$ and $S'_3=\emptyset$.} Define a total coloring
$c:V(D_{lex}(P_n,G))\cup E(D_{lex}(P_n,G))
\longrightarrow \bigcup_{i=1}^{3}S_i\cup S'_2\cup S'_3$ as follows: 

\begin{enumerate}
    \item Since $\chi''_{a}(G)\leq \Delta(G)+2$, consider an AVD-total coloring of $G_{v_1}$ with colors in $S_1$. 
    The edges between $G_{v_{1}}$ and $G_{v_{2}}$ are colored with the colors in $S_2\cup S'_2$.
    \item Similar to (1), consider 
    an AVD-total coloring of $G_{v_2}$ with colors in $S_3$ and color the edges between
    $G_{v_{2}}$ and $G_{v_{3}}$ using the colors in $S'_3\cup S_1$. 
    
    \item Assign an AVD-total coloring to $G_{v_3}$ using colors from the set $S_2$
    and assign the colors in the set $S'_2\cup S_3$ to the edges between $G_{v_3}$ and $G_{v_4}$.
    
    \item Consider 
    an AVD-total coloring of $G_{v_4}$ with colors in $S_1$ and color the edges between $G_{v_4}$ and $G_{v_5}$ with the colors in the set $S_2\cup S'_3$.

    \item Using the colors in the set $S_3$, consider 
    an AVD-total coloring of $G_{v_5}$
    and use colors in the set $S_1\cup S'_2$ for the edges between $G_{v_5}$ and $G_{v_6}$.

    \item Assign an AVD-total coloring to $G_{v_6}$ using colors from the set $S_2$ and assign the colors in the set $S_3\cup S'_3$ to the edges between $G_{v_6}$ and $G_{v_7}$.

    \item Color the vertices and edges of $G_{v_{7}}$ as in $G_{v_{1}}$, and repeat the previous steps till all the elements of $D_{lex}(P_n, G)$ are colored.
\end{enumerate}

It is clear that the coloring $c$ is a proper total coloring of $D_{lex}(P_n,G)$. 
We show that $c$ is an AVD-total coloring of $D_{lex}(P_n,G)$. 
Let $xy\in E(D_{lex}(P_n,G))$. 

\textbf{Case (i).} Suppose $x=(v_i,a)\in V(G_{v_i})$ and $y=(v_i,b)\in V(G_{v_i})$ for some $a,b\in V(G)$ and $v_{i} \in V(P_n)$.
Since we considered an AVD-total coloring of $G_{v_i}$, we obtain $C_{G_{v_i}}(x)\neq C_{G_{v_i}}(y)$.  Moreover, $C_{G_{v_i}}(x), C_{G_{v_i}}(y)\subseteq S_k$ for some $k\in \{1,2,3\}$.
Without loss of generality, let $k=1$. Then
\[
C_{D_{lex}(P_n, G)}(x)\setminus C_{G_{v_i}}(x),\;
C_{D_{lex}(P_n, G)}(y)\setminus C_{G_{v_i}}(y)
\subseteq S_2\cup S'_2\cup S_3 \cup S'_3,
\]
where $S_1 \cap (S_2\cup S'_2\cup S_3 \cup S'_3)=\emptyset$.
Consequently, $C_{D_{lex}(P_n, G)}(x)\neq C_{D_{lex}(P_n, G)}(y)$.
 
\textbf{Case (ii).} Suppose $x=(v_i,a)\in V(G_{v_i})$ and $y=(v_j,b)\in V(G_{v_j})$ for some $a,b\in V(G)$ and $v_i,v_j\in V(P_n)$ such that $i< j<n$. By the definition of $c$, the elements of $G_{v_i}$ are colored with the colors in $S_k$ for some $k\in\{1,2,3\}$.
Among the incident edges of $x$ in $D_{lex}(P_n, G)$, only the edges of $G_{v_i}$ incident to $x$ are colored with the colors in $S_k$. Since $x$ is also colored with the colors in $S_k$, we obtain $|S_k\cap C_{D_{lex}(P_n, G)}(x)|\leq \Delta(G)+1$.
However, $v_iv_j\in E(P_n)$ since $xy\in E(D_{lex}(P_n,G))$. Thus, $j=i+1$ as we assumed $i<j$. 
By the definition of $c$, all the colors in $S_k$ are used to color the edges between $G_{v_j}$ and $G_{v_{j+1}}$.
Consequently, 
\[|S_k\cap C_{D_{lex}(P_n, G)}(y)|= |S_k| = \Delta(G)+2>|S_k\cap C_{D_{lex}(P_n, G)}(x)|.\] 
Thus, there exists $a\in S_k$ such that $a\in C_{D_{lex}(P_n, G)}(y)$ but  $a\not\in C_{D_{lex}(P_n, G)}(x)$. Hence $C_{D_{lex}(P_n, G)}(x)\neq C_{D_{lex}(P_n, G)}(y)$.

\textbf{Case (iii).} Suppose $x=(v_{n-1},a)\in V(G_{v_{n-1}})$ and $y=(v_n,b)\in V(G_{v_n})$ for some $a,b\in V(G)$. If the elements of $G_{v_n}$ are colored with the colors in $S_k$ for some $k\in\{1,2,3\}$, then similar to the arguments of Case (ii), we have $|S_k\cap C_{D_{lex}(P_n, G)}(y)|\leq \Delta(G)+1$ whereas $|S_k\cap C_{D_{lex}(P_n, G)}(x)|= \Delta(G)+2$. Thus, $C_{D_{lex}(P_n, G)}(x)\neq C_{D_{lex}(P_n, G)}(y)$.
\end{proof}
\section{Applications of Latin squares} 

\begin{defn}\label{Definition 4.1}
A {\em Latin square} of order $k$ is a $k \times k$ matrix based on the elements $1, 2,...,k$ such that each element occurs exactly once in each row and exactly once in each column. A Latin square $M = [m_{i,j}]$ of order $k$ is said to be {\em commutative} if $m_{i,j} = m_{j,i}$, for $1 \leq i, j \leq k$ and $M$ is said to be {\em idempotent} if $m_{i,i} = i$, for $1 \leq i \leq k$. If the rows of $M$ are just cyclic permutations (one shift of the elements to the left) of the previous row, then $M$ is said to be {\em anti-circulant}.
\end{defn}

\begin{defn}\label{Definition 4.2}
Fix an integer $k$. We say that a graph $G=(V(G), E(G))$ is {\em $k$-edge choosable} if for any assignment $L = (L(e))_{e\in E(G)}$ of lists of available colors to the edges of $G$, there is a proper edge coloring $f$ of $G$ such that $f(e) \in L(e)$ and $\vert L(e)\vert=k$ for all $e\in E(G)$. The {\em list chromatic index}
of $G$, denoted by $\chi_{L}'(G)$, is the minimum integer $k$ such that $G$ is $k$-edge choosable.  
\end{defn}

\begin{fact}\label{Fact 4.3}
The following holds:
\begin{enumerate}
    \item (Borodin, Kostochka, and Woodall; \cite[Theorem 3]{BKW1997}) Let $G$ be a bipartite multigraph. If every edge $e=uw\in E(G)$ is assigned a list $L(e)$ of at least $\max\{deg_G(u),deg_G(w)\}$ colors, then $G$ admits a proper edge coloring $f$ such that $f(e)\in L(e)$ for every $e\in E(G)$.
 
    \item An idempotent commutative Latin square (ICLS) of order $k$ exists if and only if $k$ is odd.
    
    \item  If $M(k) = [m_{i,j}]$ is a Latin square where $m_{i,j} \equiv (i + j)k \pmod{2k - 1}, 1 \leq m_{i,j} \leq 2k - 1$, for $1 \leq i, j \leq 2k - 1$, then $M(k)$ is an anti-circulant ICLS of order $2k - 1$. 

    \item If $M = [m_{i,j}]$ is a Latin square where $m_{i,j} \equiv (i + j)-1 \pmod{n}, 1 \leq m_{i,j} \leq n$, for $1 \leq i, j \leq n$, then $M$ is an anti-circulant commutative Latin square of order $n$. 
\end{enumerate}
\end{fact}

\begin{fact}[Vizing's Theorem]\label{Fact 4.4}
For any simple graph $G$, $\Delta(G)\leq \chi'(G) \leq \Delta(G)+1$.
\end{fact}

\subsection{Join of two graphs}

\begin{thm}\label{Theorem 4.5}
Let $G_1$ and $G_2$ be graphs with $|V(G_1)|=n$ and $|V(G_2)|=m$, where $m>n+1$ and $\Delta(G_1)\geq\Delta(G_2)$. If $G_1$ satisfies the TCC, the join $G=G_1\vee G_2$ satisfies the AVD-TCC.
\end{thm}

\begin{proof}
Let $V(G_1)=\{u_1,\dots,u_n\}$ and $V(G_2)=\{v_1,\dots,v_m\}$. 
Then $\Delta(G)=\max\{\Delta(G_1)+m,\,\Delta(G_2)+n\}=\Delta(G_1)+m$.
Consider a Latin square $M=[m_{i,j}]$ of order $m+1$ where $m_{i,j}\equiv (i+j)-1 \pmod{m+1}$
as in Fact \ref{Fact 4.3}(4). 
Let $M'=[m_{i,j}]_{1\leq i\leq m,1\leq j\leq n+1}$ be the submatrix of $M$ consisting of the first $m$ rows and the first $n+1$ columns (see Figure \ref{Figure 3}).
We define
\begin{align*}
g(x)=
\begin{cases}
m_{i,j} & \text{if } x=v_iu_j \text{ with }  1\leq i \leq m, 1\leq j \leq n \\[2mm]
m_{i,n+1} & \text{if } x=v_i, 1\leq i \leq m.
\end{cases}
\end{align*}

\begin{figure}[H]
\centering
\renewcommand{\arraystretch}{0.9}
\[
M =
\begin{NiceArray}{*{11}{c}}[hvlines]
\Block[draw,line-width=1.5pt]{5-6}{}
m_{1,1} & \dots & m_{1,j} & \dots & m_{1,n} &\colorbox{gray!30}{$m_{1,n+1}$} 
& m_{1,n+2} & \dots & m_{1,m+1} \\
\vdots & \vdots & \vdots & \vdots & \vdots
& \vdots & \vdots & \vdots & \vdots \\
m_{i,1} & \dots & m_{i,j} & \dots & m_{i,n}
& \colorbox{gray!30}{$m_{i,n+1}$} & m_{i,n+2} & \dots & m_{i,m+1} \\
\vdots & \vdots & \vdots & \vdots & \vdots
& \vdots & \vdots & \vdots & \vdots \\
m_{m,1} & \dots & m_{m,j} & \dots & m_{m,n}
& \colorbox{gray!30}{$m_{m,n+1}$} & m_{m,n+2} & \dots & m_{m,m+1} \\
m_{m+1,1} & \dots & m_{m+1,j} & \dots & m_{m+1,n}
& m_{m+1,n+1} & m_{m+1,n+2} & \dots & m_{m+1,m+1}
\end{NiceArray}
\]
\vspace{-4mm}
\captionsetup{width=\textwidth}
\caption{
\em The highlighted last-column entries of $M'$ are assigned to the vertices of $G_2$.
}
\label{Figure 3}
\end{figure}

Let $B$ be the complete bipartite graph with bipartitions $V(G_1)$ and $V(G_2)$, which is a subgraph of $G$ and $C$ be the set of entries in $M$.


\begin{claim}\label{Claim 4.6}
The following holds:
\begin{enumerate}
    \item no two adjacent edges in $B$ receive the same color,
     
    \item each vertex $v_i$ receives a color distinct from its incident edges in $B$,
    
    \item Let $D_{B}(u)=\{g(uv):uv\in E(B)\}$ be the set of colors on the incident edges of a vertex $u$ in $B$. Then, $D_{B}(u_i)\neq D_{B}(u_j)$ and $C_{B}(v_i)\neq C_{B}(v_j)$ for any $i\neq j$.
\end{enumerate}
\end{claim}

\begin{proof}
Since each element of $M$ occurs exactly once in each row and column of $M$, (1) and (2) follow.
In order to prove (3), it suffices to show that no two rows of $M'$ contain the same set of entries; an analogous argument applies to the columns.
Suppose for the sake of contradiction,
$
\{m_{i,1}, m_{i,2}, \dots, m_{i,n+1}\}
=
\{m_{r,1}, m_{r,2}, \dots, m_{r,n+1}\}
$
for $i\neq r$. Without loss of generality, assume that $i>r$.
Let $t = m+1$. By the definition of $M$,
$
m_{i,j} \equiv (i + j)-1 \pmod{t}.
$
Since $m_{i,1}\in \{m_{r,1},...,m_{r,n+1}\}$, there must exist some index $j \in \{1, \dots, n+1\}$ such that $m_{i,1}=m_{r,j}$.
Thus, 
$
(i+1)-1 \equiv (r + j)-1 \pmod{t}.
$
Hence
$j \equiv i - r + 1 \pmod{t}.
$
Since $1 \le r<i \le m$ and $m < t$, we obtain $j \le m < t$, which implies
$
j = i - r + 1.
$
Since $M$ is anticirculant, we have
$
m_{i,1}=m_{r,j},\; m_{i,2}=m_{r,j+1},\; \dots,\; m_{i,n+2-(j-1)}=m_{r,n+2}.
$
Since $i>r$, we have
$
j=i-r+1\ge2.
$
Thus, $n+2-(j-1)\le n+1$, and hence $m_{r,n+2}\in\{m_{i,1},\dots,m_{i,n+1}\}$.
However,
$
m_{r,n+2}\notin\{m_{r,1},\dots,m_{r,n+1}\},
$
since every symbol appears exactly once in each row of the Latin square.\footnote{By assumption, $n+2<m+1$. Thus, the entry $m_{r,n+2}$ exists in $M$ but does not belong to $M'$.}
This contradicts the assumption that
$
\{m_{i,1},\dots,m_{i,n+1}\}
=
\{m_{r,1},\dots,m_{r,n+1}\}.
$
So, no two rows of $M'$ contain the same set of entries.
\end{proof}

Let $C$ and $D$ be disjoint color sets with $|D| = \Delta(G_1) + 2$. 
Since $G_1$ satisfies the TCC, there is a proper total coloring $f_1 : V(G_1)\cup E(G_1) \to D$ of $G_1$.
Let $f_2 : E(G_2) \to D$ be a proper edge coloring of $G_2$ 
(which exists by Vizing's Theorem (Fact \ref{Fact 4.4})) since $\Delta(G_2) \leq \Delta(G_1)$.
We now define a total coloring $f:V(G)\cup E(G)\rightarrow C\cup D$ by
\[
f(x)=
\begin{cases}
f_1(x), & \text{if } x\in V(G_1)\cup E(G_1),\\
f_2(x), & \text{if } x\in E(G_2),\\
g(x), & \text{if } x\in E(B)\cup V(G_2).
\end{cases}
\]
It is clear that $C\cap D=\emptyset$ and $|C\cup D|=(m+1)+(\Delta(G_1)+2)=\Delta(G)+3$.

\begin{claim}\label{Claim 4.7}
The coloring $f$ is an AVD-total coloring of $G$.
\end{claim}

\begin{proof}
By construction, 
$f$ is a proper total coloring of $G$.
We show that $C_{G}(x)\neq C_{G}(y)$ for every edge $xy\in E(G)$. 
For all distinct $v_i, v_j\in V(G_2)$, there exists $x\in C$ such that $x\in C_{B}(v_i)$ but $x\not\in C_{B}(v_j)$ by Claim \ref{Claim 4.6}(3). 
Since $C_{G}(v_i)\backslash C_B(v_i)\subseteq D$, $C_{G}(v_j)\backslash C_{B}(v_j)\subseteq D$ and $C\cap D=\emptyset$, we can conclude that $C_{G}(v_{i})\neq C_{G}(v_{j})$. Similarly, applying Claim \ref{Claim 4.6}(3), $C_{G}(u_{i})\neq C_{G}(u_{j})$ for all distinct $u_i,u_j\in V(G_1)$.
If $u_iv_j\in E(G)$ for any $u_i\in V(G_1)$ and $v_j\in V(G_2)$, then 
\[
|C_G(v_j)\cap C| = n+1
\quad \text{and} \quad
|C_G(u_i)\cap C| = m.
\]
By assumption $m>n+1$. Thus, $C_G(u_i)\neq C_G(v_j)$.
\end{proof}
Thus, $\chi''_a(G) \le \Delta(G)+3$, and therefore $G$ satisfies the AVD-TCC.
\end{proof}

\begin{thm}\label{Theorem 4.8}
If $G_1$ is a Type 1 graph on $n$ vertices and $G_2$ is any graph on $n$ vertices such that $\Delta(G_1)\geq \Delta(G_2)$, then $G_1 \vee G_2$ satisfies the AVD-TCC.
\end{thm}

\begin{proof}[Proof Sketch]
By the assumptions, $G_1$ has a proper total coloring with $\Delta(G_1)+1$ colors, and $G_2$ has a proper edge coloring using the same color set. 
Let $M$ be a Latin square of order $n+2$ as in Fact \ref{Fact 4.3}(4). Let $M'$ consist of the first $n$ rows and the first $n+1$ columns of $M$. Color the vertices of $G_2$ and the edges joining $G_1$ and $G_2$ using the entries of $M'$ as in the proof of Theorem~\ref{Theorem 4.5}. 
This results in an AVD-total coloring by an argument analogous to that in Claim~\ref{Claim 4.7}. 
\end{proof}

\subsection{Central graphs}

\begin{thm}\label{Theorem 4.9}
If $G$ is a connected graph of order $n\geq 5$ where $\delta(G)\geq 2$, then the following holds:
\begin{enumerate}
    \item If $n$ is even and $G\not\cong K_n$, then $C(G)$ is an AVD-Type 2 graph. 
    \item If $n$ is odd, then $C(G)$ satisfies the AVD-TCC.  
\end{enumerate}
\end{thm}

\begin{proof}
We note that $\Delta(C(G))=n-1$. 

(1). 
Since $G\not\cong K_n$, there are vertices of maximum degree that are adjacent in $C(G)$.
So, $\chi_{a}''(C(G))\geq\Delta(C(G))+2$.
We show $\chi_{a}''(C(G))\leq\Delta(C(G))+2$.
If $n$ is even, then by Fact \ref{Fact 4.3}(2,3), there exists an ICLS $M = [m_{i,j}]$ of order $n+1=\Delta(C(G))$ + 2.
Let $V_{1}=\{u_{1},...,u_{n}\}$ be the set of vertices of $G$ and $V_{2}=\{v_{1},..., v_{q}\}$ be the set of subdivision vertices in $C(G)$.
Let $M'$ consist of the first $n$ rows and the first $n$ columns of $M$.

\begin{claim}\label{Claim 4.10}
There exists a proper total coloring $g$ of $C(G)$ using the symbols from $M'$.
\end{claim}
  
\begin{proof}
First, we define a proper total coloring $\pi$ of $\overline{G}$ from $M'$. 
Let
\begin{itemize}
    \item $\pi(u_i) = m_{i,i}$ for all $1 \le i \le n$, 
    \vspace{2mm}
    
    \item $\pi(u_iu_j) = m_{i,j}$ for all $1 \leq i, j \leq n, i \neq j$ 
such that $u_iu_j \in E(\overline{G})$ (see Figure \ref{Figure 4}).
\end{itemize}
Next, we use Fact \ref{Fact 4.3}(1) 
to define a proper edge coloring
of the bipartite graph $B = C(G) \backslash E(\overline{G})$.
Let $X(u_{i})=\{m_{i,j}:1\leq j\leq n\}$ be the $i^{th}$- row of $M'$. Then for each vertex $u_{i}\in V_{1}$, there are 
\begin{center}
$n-(deg_{\overline{G}}(u_{i})+1)= \Delta(C(G))+ 1 - (deg_{\overline{G}}(u_{i})+1)=\Delta(C(G)) - deg_{\overline{G}}(u_{i})$ 
\end{center}
colors that are not used from $X(u_{i})$.
Fix a vertex $u_{i}\in V_{1}$ and an edge $u_{i}v_{j}\in E(B)$.
Define,

\begin{itemize}
    \item $V(u_{i})=\{m_{i,i}\}\cup\{m_{i,j}: u_{i}u_{j}\in E(\overline{G}), 1\leq j\leq n\}$, and 
    \vspace{2mm}
    
    \item $L(u_{i} v_{j})=X(u_{i})-V(u_{i})$.
\end{itemize}

\begin{figure}[H]
\centering
\renewcommand{\arraystretch}{0.93}
\[
M =
\begin{NiceArray}{*{7}{c}}[hvlines]
\Block[draw,line-width=1.5pt]{5-6}{}
\colorbox{gray!30}{$m_{1,1}$} & m_{1,2} & \cdots & m_{1,i} & \cdots & m_{1,n} & m_{1,n+1} \\
\vdots & \vdots & \ddots & \vdots & \ddots & \vdots & \vdots \\
m_{i,1} & m_{i,2} & \cdots & \colorbox{gray!30}{$m_{i,i}$} & \cdots & m_{i,n} & m_{i,n+1} \\
\vdots & \vdots & \ddots & \vdots & \ddots & \vdots & \vdots \\
m_{n,1} & m_{n,2} & \cdots & m_{n,i} & \cdots & \colorbox{gray!30}{$m_{n,n}$} & m_{n,n+1} \\
m_{n+1,1} & m_{n+1,2} & \cdots & m_{n+1,i} & \cdots & m_{n+1,n} & m_{n+1,n+1}
\end{NiceArray}
\]
\vspace{-4mm}
\captionsetup{width=\textwidth}
\caption{
\em The highlighted diagonal entries of $M'$ are used to color the vertices of $G$.
}
\label{Figure 4}
\end{figure}

\begin{figure}[h]
\centering
\begin{tikzpicture}[
    scale=0.57,
    big/.style={draw, rounded corners=8pt, thick},
    small/.style={draw, rounded corners=6pt},
    vertex/.style={circle, fill, inner sep=1.2pt}
]

\node[big, minimum width=3.6cm, minimum height=4.54cm] (V1) at (-2,0.5) {};
\node at (-4.5,3.5) {$V_1$};

\node[big, minimum width=2.7cm, minimum height=4.54cm] (V2) at (4.8,0.5) {};
\node at (6,3.5) {$V_2$};

\node[small, minimum width=1.5cm, minimum height=1.8cm] (T) at (-2,1.2) {};
\node at (-3.7,2.6) {$T$};

\node[small, minimum width=1.5cm, minimum height=1.2cm] (R) at (-2,-2.2) {};
\node at (-3.7,-2.5) {$R$};

\node at (0.6,-2.7) {$\overline{G}$};

\node[vertex,label=above:$u_i$] (ui) at (-2,3.5) {};

\node[vertex,label=below:$u_{i_1}$] (ui1) at (-2,2.46) {};
\node at (-2,1.5) {$\vdots$};
\node[vertex,label=below:$u_{i_k}$] (uik) at (-2,0.7) {};

\node[vertex] (r1) at (-2,-1.5) {};
\node at (-2,-2.2) {$\vdots$};
\node[vertex] (r2) at (-2,-2.9) {};

\draw[bend right=40] (ui) to (r1);
\draw[bend right=55] (ui) to (r2);

\node[vertex,label=right:$w_{u_i u_{i_1}}$] (w1) at (4.5,2) {};
\node at (4.5,0.8) {$\vdots$};
\node[vertex,label=right:$w_{u_i u_{i_k}}$] (wk) at (4.5,-2.7) {};

\draw[line width=1.2pt] (ui) -- (w1);
\draw[line width=1.2pt] (ui1) -- (w1);
\draw[line width=1.2pt] (ui) -- (wk);
\draw[line width=1.2pt] (uik) -- (wk);

\end{tikzpicture}
\caption{
\em Illustration of $C(G)$. Let $u_{i_1},\dots,u_{i_k}$ be the neighbors of $u_i$ in $G$, and let
$
R = V_1 \setminus \{u_{i_1}, \dots, u_{i_k}\}.
$
The edges $u_i u_{i_l}$, for $1 \le l \le k$, are subdivided by new vertices $w_{u_i u_{i_l}}$, and $u_i$ is adjacent to every vertex in $R$ in $C(G)$.}
\label{Figure 5}
\end{figure}

Thus, for each $u_i v_j \in E(B)$,
we have $|L(u_i v_j)| = deg_{B}(u_i)=max\{deg_{B}(u_i), deg_{B}(v_j)\}$ since $deg_{B}(v_j)=2$ as $v_j$ is a subdivision vertex, and $deg_{B}(u_i)\geq 2$ by the assumption $\delta(G)\geq 2$.
Assume that $L=(L(u_{i} v_{j}))_{u_{i}v_{j}\in E(B)}$ is an assignment of lists of available colors for the edges of the bipartite graph $B$. By Fact \ref{Fact 4.3}(1), there exists a proper edge coloring $\pi_{1}$ of $B$ such that $\pi_{1}(e)\in L(e)$ for all $e\in E(B)$. 
We define a total coloring $g$ of $C(G)$ using $n$ colors as follows: 
\begin{itemize}
    \item Let $g(u_{i}v_{j})=\pi_{1}(u_{i}v_{j})$ for all $u_{i}v_{j}\in E(B)$, 
    \vspace{2mm}
    
    \item Let $g(v)=\pi(v)$ for all $v\in V_{1}$ and $g(e)=\pi(e)$ for all $e\in E(\overline{G})$,
     \vspace{2mm}
     
    \item If $v_{k}=w_{u_{i}u_{j}}\in V_{2}$ is a vertex that subdivides the edge $u_{i}u_{j}\in E(G)$, then the set \[T(v_k)=\{m_{i,1},...,m_{i,n}\}\backslash \{g(v_{k}u_{i}),g(v_{k}u_{j}), g(u_{i}), g(u_{j})\}\] is nonempty as $n\geq 5$.
    Let $g(v_{k})$ be any color from $T(v_k)$.
\end{itemize}
By construction, $g$ is a proper total coloring of $C(G)$.
\end{proof}

\begin{claim}\label{Claim 4.11}
The coloring $g$ is an AVD-total coloring of $C(G)$.
\end{claim}

\begin{proof}
Since $M$ is anti-circulant (see Definition~\ref{Definition 4.1}), and the following observations hold, we have
$
C_{C(G)}(u_i)\neq C_{C(G)}(u_j)
$
for any two distinct vertices $u_i$ and $u_j$ in $V_1$.

\begin{enumerate}
    \item For each vertex $u_i \in V_1$, the color class $C_{C(G)}(u_i)$ is exactly the set $\{m_{i,j} : 1 \leq j \leq n\}$.
    \item The entries in the last column of $M$ are pairwise distinct.
\end{enumerate}

For all $u_{i}\in V_{1}$,
$deg_{C(G)}(u_i)=n-1\geq 4$ and for all $v_{j}\in V_{2}$, $deg_{C(G)}(v_{j})=2$. Thus, $C_{C(G)}(u_{i})\neq C_{C(G)}(v_{j})$ for each $u_{i}\in V_{1}$ and $v_{j}\in V_{2}$. Consequently, $g$ is an AVD-total coloring of $C(G)$. 
\end{proof}

(2). If $n$ is odd, then there exists an ICLS of order  $n+2=\Delta(C(G))$+ 3 by Fact \ref{Fact 4.3}(2,3).
Following the arguments of (1), we obtain an AVD-total coloring of $C(G)$ using $\Delta(C(G)) + 3$ colors. 
\end{proof}
\section{Central graphs of join of two graphs}

We begin with a short alternative proof of Theorem~\ref{Theorem 5.2}, a result of Panda et al.~\cite{PVK2020}, which will be useful in this section.
Our argument relies on the following lemma.

\begin{lem}[{\cite[Lemma 2.4]{YC1992}}]\label{Lemma 5.1}
If $G$ is a graph of order $n$ containing an independent set $S$ with $|S| \geq n - \Delta(G) - 1$, then $G$ satisfies the TCC.
\end{lem}

\begin{thm}[Panda--Verma--Keerti~\cite{PVK2020}]\label{Theorem 5.2}
The TCC holds for the central graph of any graph.
\end{thm}

\begin{proof}[Alternative proof of Theorem~\ref{Theorem 5.2}]
Let $C(G)$ be the central graph of a graph $G=(V(G),E(G))$. Let $V_1 = V(G)$ and $V_2 = V(C(G)) \setminus V(G)$, where $|V_1| = n$ and $|V_2| = r$. 
If $n \geq 3$, then $|V(C(G))| = n + r$ and $\Delta(C(G)) = n - 1$. Moreover, $V_2$ is an independent set of size
$
r = (n + r) - \Delta(C(G)) - 1.
$
By Lemma~\ref{Lemma 5.1}, $C(G)$ satisfies the TCC.
If $n \leq 2$, the result follows by direct verification.
\end{proof}

\begin{fact}\label{Fact 5.3}
The following holds:

\begin{enumerate}
    \item {(Zhang et al. \cite[Theorem 2.2]{ZCYLW2005})} If $n\geq3$, then $\chi''_{a}(K_n)=n+1$ if $n$ is even and $\chi''_{a}(K_n)=n+2$ if $n$ is odd.
    
    \item {(Panda et al. \cite[Theorems 4.1, 4.2]{PVK2020})} If $G\in \{P_{n}, C_{n}\}$, then $\chi''_{a}(C(G))=n+1$ if $n$ is even and $\chi''_{a}(C(G))\leq n+2$ otherwise.

   \item {(Panda et al. \cite[Theorems 4.3, 4.4]{PVK2020})} $\chi''_{a}(C(K_{1,n}))=n+2$ and $\chi''_{a}(C(K_{n}))=n$ for any integer $n\geq 1$.
\end{enumerate}
\end{fact}

\begin{prop}\label{Proposition 5.4}
Let $G_1$ and $G_2$ be two graphs of order at least $3$. The following holds:
\begin{enumerate}
    \item If $C(G_{1})$ and $C(G_{2})$ satisfy the AVD-TCC, then $C(G_{1} \vee G_{2})$ satisfies the AVD-TCC.
    \item The graph $C(K_{m,n})$ satisfies the AVD-TCC for any $m,n\geq 3$.
    \item If $G_{1}\in\{P_{n},C_{n},K_{n}\}$ and
    $G_{2}\in\{P_{m},C_{m},K_{m}\}$, then
    $C(G_{1}\vee G_{2})$ satisfies the AVD-TCC for any integers $m,n\geq3$.
    \item If $G_{1}$ and $G_{2}$ have the same order, then $C(G_{1} \vee G_{2})$ satisfies the AVD-TCC.
\end{enumerate}
\end{prop}

\begin{proof}
(1). Let $V=\{v_{i}: 1 \leq i \leq n\}$ and $U=\{u_{j} : 1 \leq j \leq m\}$ be the set of vertices of $G_{1}$ and $G_{2}$ respectively. 
Let $G_{n}=C(G_{1})$, $G_{m}=C(G_{2})$, and $G=C(G_{1} \vee G_{2})$. 
Without loss of generality, we assume that $m\geq n$. We note that $\Delta(G)=m+n-1$, $\Delta(G_{n})=n-1$, and $\Delta(G_{m})=m-1$. 
Define $w_{v_{p}u_{q}}$ as the vertex that subdivides the edge $v_{p}u_{q}$ for some $1\leq p\leq n$ and $1\leq q\leq m$. 
Let $C=\{1,...,m+n+2\}$ be a set of colors.
Let $f_{m}: V(G_{m})\cup E(G_{m})\rightarrow \{1,...,m+2\}$ and $f_{n}: V(G_{n})\cup E(G_{n})\rightarrow \{m+1,...,m+n+2\}$ be the AVD-total colorings of $G_{m}$ and $G_{n}$ respectively. This is possible since $G_{m}$ and $G_{n}$ satisfy the AVD-TCC.
We define the coloring $f: V(G)\cup E(G) \rightarrow \{1,...,m+n+2\}$ as follows:
\[
f(x)=
\begin{cases}
f_n(x), & \text{if } x\in V(G_n)\cup E(G_n),\\
f_m(x), & \text{if } x\in V(G_m)\cup E(G_m),\\
m+2+q, & \text{if } x=u_iw_{v_qu_i},\\
i, & \text{if } x=w_{v_qu_i}v_q,\\
k, & \text{if } x=w_{v_i u_j},\;
k\in C\setminus\{f(w_{v_i u_j}u_j),f(v_iw_{v_i u_j}),f(v_i),f(u_j)\}
\end{cases}
\]
for each $1\leq i\leq m$ and $1\leq q\leq n$.\footnote{Since $m,n \ge 3$, we have $|C| = m+n+2 > 4$. So, the set 
$C \setminus \{f(w_{v_i u_j}u_j), f(v_i w_{v_i u_j}), f(v_i), f(u_j)\}$ 
is nonempty.} By construction, $f$ is a proper total coloring of $G$.

\begin{figure}[h]
\centering
\begin{tikzpicture}[scale=0.9]

\node at (3,3) {$W$};

\draw[rounded corners] (-1,2.7) rectangle (1,-0.7);
\node at (0,3) {$G_{n}=C(G_1)$};

\draw[rounded corners] (5,2.7) rectangle (7,-0.7);
\node at (6,3) {$G_{m}=C(G_2)$};

\fill (0,2.5) circle (1.2pt) node[left=4pt] {$v_1$};
\fill (0,1.3) circle (1.2pt) node[left=4pt] {$v_2$};
\node at (0,0.1) {$\vdots$};
\fill (0,-0.5) circle (1.2pt) node[left=4pt] {$v_n$};

\fill (6,2.5) circle (1.2pt) node[right=4pt] {$u_1$};
\fill (6,1.3) circle (1.2pt) node[right=4pt] {$u_2$};
\node at (6,0.1) {$\vdots$};
\fill (6,-0.5) circle (1.2pt) node[right=4pt] {$u_m$};

\draw (0,2.5) -- (0,1.3);
\draw[bend left=-70] (0,2.5) to (0,-0.5);

\draw (6,2.5) -- (6,1.3);
\draw[bend right=-70] (6,2.5) to (6,-0.5);

\fill (3,2.7) circle (1.2pt) node[below left=0pt] {$w_{v_{1}u_{1}}$};
\fill (3,2.0) circle (1.2pt) node[below left=0pt] {$w_{v_{1}u_{2}}$};
\fill (3,1.2) circle (1.2pt) node[below left=0pt] {$w_{v_{2}u_{1}}$};
\fill (3,0.5) circle (1.2pt) node[below left=0pt] {$w_{v_{2}u_{2}}$};
\node at (3,0.1) {$\vdots$};
\fill (3,-0.5) circle (1.2pt) node[below left=0pt] {$w_{v_{n}u_{m}}$};


\draw (0,2.5) -- (3,2.7);
\draw (3,2.7) -- (6,2.5);

\draw (0,2.5) -- (3,2.0);
\draw (3,2.0) -- (6,1.3);

\draw (0,1.3) -- (3,1.2);
\draw (3,1.2) -- (6,2.5);

\draw (0,1.3) -- (3,0.5);
\draw (3,0.5) -- (6,1.3);

\draw (0,-0.5) -- (3,-0.5);
\draw (3,-0.5) -- (6,-0.5);

\end{tikzpicture}
\vspace{-2mm}
\caption{\em A representation of $C(G_1\vee G_2)$.}
\label{Figure 6}
\end{figure}


\begin{claim}\label{Claim 5.5}
The coloring $f$ is an AVD-total coloring of $G$.
\end{claim}
 
\begin{proof}
We show that $C_G(x)\neq C_G(y)$ for every $xy\in E(G)$. For each $u_i\in U$ and $v_j\in V$, define
\[
U(u_i)=\{f(u_iw_{v_qu_i}) : 1\le q\le n\},\qquad
V(v_j)=\{f(w_{v_ju_q}v_j) : 1\le q\le m\}.
\]
Then $U(u_k)=U(u_l)$ for all distinct $u_k,u_l\in U$, and $V(v_k)=V(v_l)$ for all distinct $v_k,v_l\in V$.
Since $f_{n}$ and $f_{m}$ are AVD-total colorings of $G_{m}$ and $G_{n}$ respectively, we have $C_{G_{n}}(v_{i})\neq C_{G_{n}}(v_{j})$ and $C_{G_{m}}(u_{i})\neq C_{G_{m}}(u_{j})$ if $v_{i} v_{j}\in E(G_{n})$ and $u_{i} u_{j}\in E(G_{m})$. Consequently, $C_{G}(v_{i})=C_{G_{n}}(v_{i}) \cup V(v_{i})\neq C_{G_{n}}(v_{j})\cup V(v_{j})=C_{G}(v_{j})$. Similarly, $C_{G}(u_{i})\neq C_{G}(u_{j})$.

Let $W=\{w_{v_pu_q}: 1\leq p\leq n,\; 1\leq q\leq m\}$ be the set of subdivision vertices corresponding to the edges joining $G_1$ and $G_2$. All vertices of $V$ and $U$ have degree $m+n-1> 2$, and all vertices of $W$ have degree 2. So, $C_{G}(w_{v_{i}u_{j}})\not\in \{C_{G}(v_{i}),C_{G}(u_{j})\}$ for any $1\leq j\leq m$ and $1\leq i \leq n$.
\end{proof}

(2). Let $G_{1}$ and $G_{2}$ be independent sets of sizes $m$ and $n$, respectively. Then $C(G_{1})\cong K_{m}$, $C(G_{2})\cong K_{n}$ and $G_{1} \vee G_{2}\cong K_{m,n}$. The rest follows from (1) and Fact \ref{Fact 5.3}(1).

(3). This follows from (1) and Fact \ref{Fact 5.3}(2,3).

(4). Let $m=n$ and assume $V, U, G_{m}, G_{n},$ and $w_{v_{p}u_{q}}$ as in the proof of (1). Let $C=\{1,...,2n+2\}$. Let $f_{m}: V(G_{m})\cup E(G_{m})\rightarrow \{1,...,n+1\}$ and $f_{n}: V(G_{n})\cup E(G_{n})\rightarrow \{n+2,...,2n+2\}$ be the proper total colorings of $G_{m}$ and $G_{n}$ respectively. This follows from Theorem \ref{Theorem 5.2}. Define $f: V(G)\cup E(G) \rightarrow \{1,...,2n+2\}$ as follows:
\begin{align*}
f(x)=
\begin{cases}
f_n(x), & \text{if } x\in V(G_n)\cup E(G_n),\\
f_m(x), & \text{if } x\in V(G_m)\cup E(G_m),\\
n+1+q, & \text{if } x=u_iw_{v_qu_i},\ q\neq i,\\
2n+2, & \text{if } x=u_iw_{v_iu_i},\\
q, & \text{if } x=w_{v_qu_i}v_q,\ q\neq i,\\
n+1, & \text{if } x=w_{v_iu_i}v_i,\\
k, & \text{if } x=w_{v_iu_j},\;
k\in C\setminus
\{f(w_{v_iu_j}u_j),f(v_iw_{v_iu_j}),f(v_i),f(u_j)\}
\end{cases}
\end{align*}
where $1\leq i,q\leq n$. 
We note that the color $i$ is missing in $C_{G}(v_{i})$ and the color $n+1+i$ is missing in $C_{G}(u_{i})$ for each $1\leq i\leq n$. So, $C_{G}(v_{i})\neq C_{G}(v_{j})$ and $C_{G}(u_{i})\neq C_{G}(u_{j})$ if $v_{i} v_{j}\in E(G_{n})$ and $u_{i} u_{j}\in E(G_{m})$. Thus, in view of the proof of Claim \ref{Claim 5.5}, $f$ is an AVD-total coloring.
\end{proof}




\end{document}